\documentclass[a4paper,11pt,reqno]{amsart}
\usepackage{amsmath,amsfonts,amssymb,amsthm,graphicx,color}

\numberwithin{equation}{section}

\def\eps{\varepsilon}

\def\H{\mathcal{H}}
\def\la{\lambda}
\def\lb{\lambda}

\def\N{\mathbb{N}}

\def\R{\mathbb{R}}

\def\Om{\Omega}

\def\bal{\begin{aligned}}
\def\eal{\end{aligned}}

\newcounter{mt}

\def\maintheoremdeclaration#1{\stepcounter{mt}\newcounter{#1}\setcounter{#1}{\arabic{mt}}}

\maintheoremdeclaration{mainintro}
\maintheoremdeclaration{main}
\maintheoremdeclaration{bounded}

\newtheorem{theorem}{Theorem}[section]
\newtheorem{lemma}[theorem]{Lemma}

\newtheorem{remark}[theorem]{Remark}
\newtheorem{corollary}[theorem]{Corollary}

\def\diam{\mathrm{diam}\,}

\def\eps{\varepsilon}

\def\p{\mathrm{Per}}

\title{A surgery result for the spectrum of the Dirichlet Laplacian}
\author{Dorin Bucur}\address{Institut Universitaire de France and Laboratoire de Math\'ematiques (LAMA) UMR CNRS 5127, Universit\'e de Savoie, Campus Scientifique, 73376 Le-Bourget-Du-Lac - FRANCE}\email{dorin.bucur@univ-savoie.fr}
\author{Dario Mazzoleni}\address{Dipartimento di Matematica, Universit\`a degli Studi di Pavia, via Ferrata, 1, 27100 Pavia - ITALY}
\email{dario.mazzoleni@unipv.it}

\begin{document}
\begin{abstract}
In this paper we give a method to geometrically modify an open set such that the first $k$ eigenvalues of the Dirichlet Laplacian and its perimeter are not increasing, its measure remains constant, and both perimeter and diameter decrease below a certain threshold. The key point of the analysis relies on the properties of the shape subsolutions for the torsion energy.
\end{abstract}
\maketitle

\section{Introduction and statement of the main results}
The results of this paper are motivated by spectral shape optimization problems 
 for the eigenvalues of the Dirichlet Laplacian, e.g.
\begin{equation}\label{0}
\min{\left\{\la_k(\Om),\;\Om\subset\R^N,\;|\Om|=1\right\}},
\end{equation}
where $\la_k$ denotes the $k^{th}$ eigenvalue of the Dirichlet Laplacian and $|\cdot|$ the $N$ dimensional Lebesgue measure ($N\ge 2$).
In order to prove existence of an optimal set $\Om$ for problem~\eqref{0}, two different methods were proposed recently. On the one hand, in \cite{mp} it is proved a surgery result asserting that one can suitably modify an open set such that the first $k$ eigenvalues of the Dirichlet Laplacian are not increasing, its measure remains constant and its diameter decrease below a certain threshold. This result together to the Buttazzo-Dal Maso existence theorem~\cite{BM} (which has a local character) gives a proof of global existence of solutions. By a different method, based on the so called shape subsolutions (see the definition in Section 2), in  \cite{blak} is
proved the existence of solutions and moreover that all minimizers have finite diameter  and finite perimeter. 

Recently,   Van den Berg has studied  in~\cite{vb} a minimum problem with both a measure \emph{and} a perimeter constraint:
\begin{equation}\label{1}
\min{\left\{\la_k(\Om),\;\Om\subset\R^N,\;|\Om|\le1,\;\p(\Om)\leq C\right\}}.
\end{equation}
An existence result for this problem cannot be deduced from the results~\cite{blak, mp}. The surgery method of \cite{mp} can hardly control  the perimeter since the procedure generates new pieces of boundary which may have a large surface area. As well, in the presence of two simultaneous constraints, the notion of  shape subsolution can not be used in a direct manner due to the lack of suitable Lagrange multipliers which can take into account {\it both } geometric constraints. The results of this paper are also intended to provide a tool allowing to prove existence of a solution for \eqref{1}.

In this paper we  give a result which follows the main objectives of  \cite{mp}, but with the new requirement on the control of  the perimeter. For this purpose, the ``surgery" is done in a different manner, using some of the key ideas of the shape subsolutions. Roughly speaking, we look at the local behavior of the torsion function and prove that if this function is small enough in some region, then one can cut out a piece of the domain controlling simultaneously the variation of the low part of the spectrum, of the measure and of the perimeter.

Throughout the paper, by $\tilde \Om$ we denote an open set of finite measure. For simplicity, and without restricting the generality, we shall assume that its measure is equal to $1$. 
Here is our main result which, for clarity, is stated in a simplified way: 

\begin{theorem}\label{mainintro}
For every $K>0$, there exists $D,C>0$ depending only on $K$ and the dimension $
N$, such that for every open set $\tilde \Om\subset \R^N$ with $|\tilde \Om|=1$ there exists an open set $\Om$ satisfying 
\begin{enumerate}
\item $|\Om|=1$, $\diam(\Om)\leq D$ and  $\p(\Om)\leq \min \{\p(\tilde \Om), C\}$,
\item if  $\la_k(\tilde \Om)\leq K$, then  $\la_k(\Om)\leq \la_k(\tilde \Om)$.
\end{enumerate}
\end{theorem}
The set $\Om$ is essentially obtained by removing some parts of $\tilde \Om$ and rescaling it to satisfy the measure constraint. 
In case the measure of $\tilde \Om$ is not equal to $1$, the constants $D$ and $C$ above depend also on $|\tilde \Om|$, following the rescaling rules of the eigenvalues, measure and perimeter

We shall split the main result stated above in two distinct theorems, Thereoms \ref{bounded} and \ref{main}. The construction of $\Om$ differs depending on which kind of control of the perimeter is desired. If the perimeter of $\tilde \Om$ is infinite (or larger than $C$), it is convenient to use an optimization argument related to the shape subsolutions  to directly construct the set $\Om$ satisfying all the requirements above on eigenvalues, measure and diameter, but with a perimeter less than $C$ (Theorem \ref{bounded}). If the  perimeter of $\tilde \Om$ is finite (for example smaller than $C$), we produce a different argument, by cutting in a suitable way the set $\tilde \Om$ with hyperplanes, and removing some strips, decreasing in this way the perimeter (Thereom \ref{main}) and of course satisfying all the requirements above on eigenvalues, measure and diameter.  In this last case,  the control of the perimeter is done through a De Giorgi type argument. We point out that the assertions of the two theorems are slightly stronger than the unified formulation stated in Theorem \ref{mainintro}.

We note that the results of this paper hold true in exactly the same way if instead of ``open" sets one works with ``quasi-open" or  ``measurable" sets (see the precise definitions of the spectrum for this weaker settings in \cite{bmpv14}). In general, if $\tilde \Om$ is quasi-open or measurable, then the constructed set $\Om$ is of the same type. In some situations in which the diameter of $\tilde  \Om$ is  large, $\Om$ could be chosen open and smooth.

\section{The spectrum of the Dirichlet Laplacian and the torsion function}
Let $\Om \subset \R^N$ be an open set of finite measure. Denoting by $H^1_0(\Om)$ the usual Sobolev space, the eigenvalues of the Dirichlet Laplacian on $\Om$ are defined by
\begin{equation}
\lambda_k(\Om):=\min_{S_k}\max_{u\in S_k\setminus \{0\}}\frac{\int_\Om |\nabla u|^2\,dx}{\int _\Om u^2\,dx}\,,
\end{equation}
 where the minimum ranges over all $k$-dimensional subspaces $S_k$ of $H^1_0(\Om)$. 

The \emph{torsion} function of  $\Omega$ is the function denoted $w_\Om$  which minimizes the {\it torsion} energy
\[
E(\Om):=\min_{u\in H^1_0(\Om)}{\frac{1}{2}\int_{\R^N}{|Du|^2}dx-\int_{\R^N}{u}dx},
\]
and satisfies in a weak sense 
\[
-\Delta w_{\Om}=1\quad \mbox{in }\Om, \qquad w_\Om\in H^1_0(\Om).
\]
Note that the torsion energy is negative if $\Om\not=\emptyset$ and 
$$E(\Om) = -\frac12\int_{\R^N}{w_{\Om}}dx< 0.$$
We recall  (see for instance \cite{BM}) that if one extends the torsion function by zero on $\R^N \setminus \Om$, then it satisfies $-\Delta w_\Om\le 1$ in the sense of distributions in $\R^N$. 

A fundamental property of the torsion function is the Saint Venant inequality, which states that among all (open) sets of equal volume, the ball maximizes the $L^1$-norm of the torsion function. This leads to the following inequality
\begin{equation}\label{saintvenant}
\int_{\Om}{w_\Om}dx\leq |\Om|^{\frac{N+2}{N}}\frac{\omega_N^{-2/N}}{N(N+2)},
\end{equation}
where $\omega_N$ is the volume of the ball of \emph{radius} $1$ in $\R^N$.
A similar  inequality between the $L^\infty$ norms  was proved by Talenti in \cite{talenti} and leads to :
\begin{equation}\label{bdm03}
\|w_{\Om}\|_{L^\infty}\leq   \Big(\frac{|\Om|}{\omega_N}\Big ) ^\frac 2N \frac{1}{2N}.
\end{equation}

We recall the following bound on the ratio between eigenvalues of the Dirichlet Laplacian, which can be found in~\cite{A}. For all $k\in \N$ there exists a constant $M_k$, depending \emph{only} on $k$ and the dimension $N$, such that 
\begin{equation}\label{boundk}
1\le \frac{\la_k(\Om)}{\la_1(\Om)}\leq M_k.
\end{equation}
Another fundamental inequality, proved in~\cite{vb1} (see also \cite{bvb}), relates the $L^\infty$ norm of the torsion function with the first eigenvalue and reads
\begin{equation}\label{vdb}
\frac{1}{\la_1(\Om)}\leq \|w_{\Om}\|_{\infty}\leq \frac{4+3N\log 2}{\la_1(\Om)}.
\end{equation}
We also recall the following inequality due to Berezin, Li and Yau (see  \cite{ly83}), which asserts that for some constant $C_N$ depending only on the dimensions of the space, we have 
$$\forall k\in \N\;\;\; \lb_k(\Om)\ge C_N\Big ( \frac{k}{|\Om|}\Big ) ^\frac{2}{N}.$$
The way we shall use this inequality is the following: if one fixes $K >0$, then the number of eigenvalues of $\Om$ below $K$, is at most of $\Big (\frac{K}{C_N}\Big ) ^\frac N2 |\Om|$.

The $\gamma$-distance between two open sets with finite measure $\Om_1,\Om_2$ is defined by:\[
d_{\gamma}(\Om_1,\Om_2):=\int_{\R^N}{|w_{\Om_1}-w_{\Om_2}|}dx.
\]
For sets satisfying $\Om_1 \subseteq \Om_2$, the following inequality was proved in \cite{blak} : for every $k\in \N$
\begin{equation}\label{bm02}
\Big | \frac{1}{\lb_k(\Om_1)}- \frac{1}{\lb_k(\Om_2)}\Big | \le 2k^2e^{1/4\pi}\la_k(\Om_2)^{N/2} d_\gamma(\Om_1,\Om_2),
\end{equation}
and we notice that there is a strong relation between the $\gamma$-distance and the torsion energy:  $d_{\gamma}(\Om_1,\Om_2)=2(E(\Om_1)-E(\Om_2))$. 

Let $c>0$. It is said that $\tilde \Om\subset \R^N$ is a shape subsolution for the energy if for all $\Om\subset\tilde \Om$
 \[
E(\Om)+c|\Om|\geq E(\tilde \Om)+c|\tilde \Om|.
\]
It is proved in~\cite{blak} that, if $\tilde \Om$ is a  shape subsolution for the energy, then it is bounded (with controlled diameter)  and has finite perimeter.

We conclude this Section with a result relating a pointwise value of the torsion function to its integral on some neighborhood. 
\begin{lemma}\label{density}
Let $ \Om\subset\R^N$ be an open set and $w=w_{\Om}$ be its torsion function. For every $\theta>0$, 
there exists $\delta_0 >0$ depending only on $N,\theta$ such that if $w(x)\geq \theta$ for some $x\in\R^N$, then \[
\int_{B_\delta (x)}{w}dx\geq \frac{\theta\omega_N}{2}\delta^N,\qquad\forall \delta\in(0,\delta_0).
\]
\end{lemma}
\begin{proof}
Since for every $x_0 \in \R^N$ the function $x \mapsto w(x)+\frac{|x-x_0|^2}{2N}$ is subharmonic in $\R^N$, we have that, for all $\delta>0$ \[
\theta\leq w(x_0)\leq \frac{1}{|B_\delta|}\int_{B_\delta(x_0)}{(w(x)+\frac{|x-x_0|^2}{2N})\,dx}=\frac{1}{|B_\delta|}\int_{B_\delta (x_0)}{w}dx+\frac{\delta^{2}}{2(N+2)}.
\]
For some $ \delta_0 $ sufficiently small (e.g equal to $\sqrt{\theta (N+2)}$), we have $\forall\;  0< \delta\leq \delta _0$ \[
\int_{B_{\delta}(x)}{w}dx\geq\frac{\theta\omega_N}{2}\delta^N.
\]
\end{proof}

\section{Control of the spectrum by subsolutions}\label{c}
Before stating our first result, we outline the main ideas.
Let $\tilde \Om\subset \R^N$ be a given open set of finite measure. Assume that for some set $\Om\subset \tilde \Om$  and for some constant $c>0$ we have
\begin{equation}\label{bmd01}
E(\Om)+c|\Om| \le E(\tilde \Om)+c|\tilde \Om|.
\end{equation}
Then, we shall observe  that a certain number of low eigenvalues of the rescaled set $\Big ( \frac{|\tilde \Om|}{|\Om|}\Big ) ^\frac 1N\Om$ are not larger than the corresponding eigenvalues on $\tilde \Om$, provided that $c$ is small enough. 
Smaller is the constant $c$, more eigenvalues satisfy this property.  Indeed, from \eqref{bm02}, we get
\begin{equation}\label{bmd04}
\lb_k(\Om)-\lb_k(\tilde \Om) \le 4k^2e^{1/4\pi}\la_k(\Om)\la_k(\tilde \Om)^{(N+2)/2} [E(\Om)-E(\tilde \Om)]. 
\end{equation}
Setting $K_{\Om,\tilde \Om}= 4k^2e^{1/4\pi}\la_k(\Om)\la_k(\tilde \Om)^{(N+2)/2}$, 
using inequality \eqref{bmd01} we get

\begin{equation}\label{bmd05}
\lb_k(\Om)-\lb_k(\tilde \Om) \le c K_{\Om,\tilde \Om}  (|\tilde \Om|-|\Om|)\le c K_{\Om,\tilde \Om}  |\tilde \Om|^{\frac{N-2}{N}} \frac{N}{2}  ( |\tilde \Om|^{\frac{2}{N}}- |\Om|^{\frac{2}{N}}).
\end{equation}

Then, for every $\Lambda$ such that 
\begin{equation}\label{bmd10}
c K_{\Om,\tilde \Om}  |\tilde \Om|^{\frac{N-2}{N}} \frac{N}{2} \le \Lambda
\end{equation}
we get
$$\lb_k(\Om)-\lb_k(\tilde \Om)  \le \Lambda( |\tilde \Om|^{\frac{2}{N}}- |\Om|^{\frac{2}{N}}),$$
so
$$\lb_k(\Om)+\Lambda |\Om|^{\frac{2}{N}}\le \lb_k(\tilde\Om)+\Lambda |\tilde\Om|^{\frac{2}{N}}.$$
If $c$ is small enough so that we can choose $\Lambda$ satisfying $\lb_k(\tilde\Om)=\Lambda |\tilde\Om|^{\frac{2}{N}}$, we get 
$$\lb_k(\Om)|\Om|^{\frac{2}{N}}\le  \lb_k(\tilde\Om) |\tilde\Om|^{\frac{2}{N}}.$$
The construction  above can be carried out provided that one has control on an upper bound of $K_{\Om,\tilde \Om}$ in \eqref{bmd10}. We shall prove that this is the case, if $c$ is small enough.

\begin{lemma}\label{choicec}
Let $k\in \N, K>0$ and $\tilde \Om\subset \R^N$ be an open set of unit measure, satisfying $\la_k(\tilde \Om)\leq K$.
There exist two  constants $c, \beta>0$ depending only on $K$ and $N$ such that 
 for all $\Om\subset \tilde \Om$ satisfying
\begin{equation}\label{bmd02}
E(\Om)+c|\Om|\leq E(\tilde \Om)+c|\tilde \Om|
\end{equation}
 we have
 \begin{equation}\label{optresc}
\la_i(\Om)|\Om|^{2/N}\leq \la_i(\tilde \Om)|\tilde \Om|^{2/N},\qquad\forall\;i=1,\dots,k
\end{equation}
and $|\Om|\geq\beta |\tilde \Om|$.
\end{lemma}
\begin{proof}
We divide the proof in several steps.

\noindent {\it Step 1.} The constant $c$ can be chosen such that $E(\tilde \Om)+c|\tilde \Om|$ is negative.
In order to find the right information on $c$, we start by proving the following inequality: 
\begin{equation}
\int_{\tilde \Om}{w_{\tilde \Om}}dx\geq C(N)\frac{1}{(2\la_1(\tilde \Om))^{\frac{N+2}{2}}},
\end{equation}
with $C(N):=\frac{(2N)^{\frac{N+2}{2}}\omega_N}{N(N+2)}.$ 
We first note that $(w_{\tilde \Om}-\frac{1}{2\la_1(\tilde \Om)})^+$ is the torsion function of the set $\{w_{\tilde \Om}>\frac{1}{2\la_1(\tilde \Om)}\}$ and  that $\|w_{\tilde \Om}-\frac{1}{2\la_1(\tilde \Om)}\|_{\infty}\geq 1/2\la_1(\tilde \Om)$, thanks to~\eqref{vdb}. As a consequence of the Talenti inequality \eqref{bdm03}, the measure of the set $\{w_{\tilde \Om}>\frac{1}{2\la_1(\tilde \Om)}\}$ is controlled  from below by $\lb_1(\tilde \Om)$,
precisely we have
 \[
\int_{\tilde \Om}{w_{\tilde \Om}}dx\geq \int_{\left\{w_{\tilde \Om}>\frac{1}{2\la_1(\tilde \Om)}\right\}}{\left(w_{\tilde \Om}-\frac{1}{2\la_1(\tilde \Om)}\right)}dx+\int_{\left\{w_{\tilde \Om}>\frac{1}{2\la_1(\tilde \Om)}\right\}}{\frac{1}{2\la_1(\tilde \Om)}}dx\geq C(N)\frac{1}{(2\la_1(\tilde \Om))^{\frac{N+2}{2}}}.
\]
Then it is clear that we have \[
E(\tilde \Om)+c|\tilde \Om|\leq -\frac{C(N)}{2(2\la_1(\tilde \Om))^{\frac{N+2}{2}}}+c|\tilde \Om|.
\]
The right hand side above is negative, as soon as we choose 
\begin{equation}
c\leq \frac{C(N)}{2|\tilde \Om|(2K)^{\frac{N+2}{2}}},
\end{equation}
since $\la_1(\tilde \Om)\leq K$.

\noindent {\it Step 2.}  The constant $c$ can be chosen such that for every $\Om$ satisfying \eqref{bmd02}, we have
$\|w_\Om\|_\infty \ge \frac 12 \|w_{\tilde \Om}\|_\infty.$
 Indeed, denote $h:=\|w_{\tilde \Om}\|_\infty$ and assume that $\|w_{\Om}\|_\infty <  \frac{\|w_{\tilde \Om}\|_\infty}{2}$.
Then \[
\begin{split}
0&\leq c|\Om|\leq c|\tilde \Om|+\frac12\int{w_{\Om}}dx-\frac12\int{w_{\tilde \Om}}dx\\
&=c|\tilde \Om|+\frac12\int{w_{\Om}}dx-\frac12\int{\min{\left\{w_{\tilde \Om}, h/2\right\}}}dx-\frac12\int_{\left\{w_{\tilde \Om}>h/2\right\}}{\left(w_{\tilde \Om}-h/2\right)}dx\\
&\leq c|\tilde \Om|-\frac12\int_{\left\{w_{\tilde \Om}>h/2\right\}}{\left(w_{\tilde \Om}-h/2\right)}dx.
\end{split}
\]
Thanks to the fact that $(w_{\tilde \Om}-h/2)^+=w_{\left\{w_{\tilde \Om}>h/2\right\}}$ using the same argument as in Step 1, we have that \[
C(N)h^{\frac{N+2}{2}}\leq \frac12\int_{\left\{w_{\tilde \Om}>h/2\right\}} ({w_{\tilde \Om}-h/2})^+\leq c|\tilde \Om|.
\]
Consequently, if \[
c\leq C(N)K^{-\frac{N+2}{2}}
\]
then  $\|w_{\Om}\|_\infty\ge \frac{\|w_{\tilde \Om}\|_\infty}{2}$, since  by inequality~\eqref{vdb} $\|w_{\tilde \Om}\|_{\infty}\geq \frac{1}{\la_1(\tilde \Om)}\geq \frac{1}{K}$.

We note that, using the inequalities  \eqref{boundk} and \eqref{vdb}  together with the fact that $\|w_{\Om}\|_{\infty}\geq\|w_{\tilde \Om}\|_{\infty}/2$, one can easily deduce that for every $i \in \N$ the corresponding eigenvalues on $\Om$ and $\tilde \Om$ are comparable
\begin{equation}\label{la1}
\la_i(\tilde \Om) \le \la_i(\Om)\leq (8+6N \log 2) M_i \la_i(\tilde \Om).
\end{equation}

\noindent {\it Step 3.} Proof of inequality \eqref{optresc}. Choosing $c$ satisfying Steps 1 and 2, and
 \begin{equation}\label{bmd06}
c\leq \frac{\la_1(B)}{2M_kk^2 (8+6N\log 2)e^{1/4\pi}k^2K^{\frac{N}{2}+2}},
\end{equation}
where $B$ is the ball of volume equal to $1$, from the Faber-Krahn inequality we have
\begin{equation}
c2M_kk^2 (8+6N\log 2)e^{1/4\pi}k^2K^{\frac{N}{2}+2}\leq \frac{\la_1(\tilde \Om)}{|\tilde \Om|^{2/N}},
\end{equation}
 or this precisely gives, in view of \eqref{bmd04}-\eqref{bmd05},
\begin{equation}\label{bmd11}
\la_i(\Om)+\Lambda|\Om|^{2/N}\leq \la_i(\tilde \Om)+\Lambda|\tilde \Om|^{2/N}.
\end{equation}
  Thanks to Step 2, $c$ is such that $\Lambda\leq \frac{\la_1(\tilde \Om)}{|\tilde \Om|^{2/N}}$. Consequently,  \eqref{bmd11} also holds for all values larger than $\frac{\la_1(\tilde \Om)}{|\tilde \Om|^{2/N}}$, thus this inequality holds for all $i=1,\dots, k$ with constant $\frac{\la_i(\tilde \Om)}{|\tilde \Om|^{2/N}}$. 
 
With the choice of $\frac{\la_i(\tilde \Om)}{|\tilde \Om|^{2/N}}$, using the arithmetic geometric inequality, we note that 
\[
\la_i(\Om)|\Om|^{2/N}\leq \la_i(\tilde \Om)|\tilde \Om|^{2/N}, \qquad\forall\;i=1,\dots, k.
\]

\noindent {\it Step 4.} In order to control the diameter of the rescaled set, we prove the existence of  $\beta>0$ such that $|\Om|\geq \beta|\tilde \Om|$. 
Indeed, we have the chain of inequalities (the last one being a consequence of \eqref{bdm03})
 \[
\frac{1}{2K}\le \frac{1}{2\la_1(\tilde \Om)}\leq \frac{\|w_{\tilde \Om}\|_{\infty}}{2}\leq \|w_{\Om}\|_{\infty} \leq  \Big(\frac{|\Om|}{\omega_N}\Big ) ^\frac 2N \frac{1}{2N},
\]
which gives the estimate for $\beta$, depending only on $K, N$.
\end{proof}
\begin{remark}
Inequality \eqref{la1} in Step 2 could also be obtained in a different way, as a consequence of inequality \eqref{bm02} by choosing $c$ small enough. Indeed, if $c$ is such that 
\begin{equation}\label{bmd50}
2k^2e^{1/4\pi}\la_k(\tilde \Om)^{N/2} d_\gamma(\Om,\tilde \Om)\le \frac{1}{2\lb_k(\tilde \Om)},
\end{equation}
then $ \lb_k(\tilde \Om)\le \lb_k (\Om)\le 2 \lb_k(\tilde \Om)$.  Since by the hypothesis of Lemma \ref{choicec} we have
$$E(\Om)-E(\tilde \Om) \le c,$$
inequality \eqref{bmd50} holds as soon as
$$  c  \le \frac{1}{8k^2e^{1/4\pi} \lb_k(\tilde \Om)^{N/2 +1}}.$$
\end{remark}

Now, we are in position to prove the first result. Below, the diameter of a disconnected set is referred as the sum of the diameters of each connected component.
\begin{theorem}\label{bounded}
{For every $K>0$, there exists $D,C>0$ depending only on $K$  and the dimension $
N$ such that for every open set $\tilde \Om\subset \R^N$ with $|\tilde \Om|=1$  there exist an open set $\Om$ with $\diam(\Om)\leq D$, $|\Om|=1$, $\p(\Om)\leq C$ and if $\la_k(\tilde \Om)\leq K$ then $\la_k(\Om)\leq \la_k(\tilde \Om)$.}
\end{theorem}

\begin{proof}
From the Berezin-Li-Yau inequality, the maximal index $k_0$  for which it is possible that $\lb_{k_0}(\tilde \Om)\le K$ is lower than a constant depending only on $K$ and $N$. 

Let us consider the minimum problem\[
\min_{\Om\subset \tilde \Om}{\left\{E(\Om)+c|\Om|\right\}},
\]
with $c$ the constant given by Lemma~\ref{choicec} and  $k= k_0$. This problem has at least one solution, denoted $\Om^*$, which is an open set (see for instance \cite{hayouni}) and it is also a shape subsolution of the energy. The results from~\cite{blak} give that $\diam( \Om^*)\leq D(c)$ and that $\p(\Om^*)\leq C(c)$ and we remind that $c$ depends only on $K$ and the dimension $N$.
Moreover, using Step 4 of Lemma~\ref{choicec}, we have that the set $\Om:=|\Om^*|^{-1/N}\Om^*$ has still diameter and perimeter bounded by constants depending only on $K,N$, thanks to the fact that $|\Om|\geq \beta |\tilde \Om|$.
Moreover we have that \[\forall i=1,\dots, k, \;\;
\la_i(\Om)\leq \la_i(\tilde \Om),
\] 
since $\lb_k(\tilde \Om) \le K$.
\end{proof}

\section{Control of the perimeter}\label{cl}

In order to give precise statements, 
we introduce a suitable notion of diameter in a prescribed direction. In the coordinate direction $e_1\in \R^N$ we set
\[
\diam_{e_1}(\Om):=\mathcal{H}^1\left(t\in\R\;:\;\mathcal{H}^{N-1}(\Om\cap\{x_1=t\})>0\right).
\]

\begin{theorem}\label{main}
{For every $K,P>0$, there exist $D>0$ depending only on $K,P$ and the dimension $
N$, such that for every open set $\tilde \Om\subset \R^N$ with $|\tilde \Om|=1$, $\p(\tilde \Om) \le P$,   there exists an open set $\Om$ of unit measure with $\diam_{e_1}(\Om)\leq D$,  $\p(\Om)\leq \p(\tilde \Om)$ such that if $\la_k(\tilde \Om)\leq K$ then $\la_k(\Om)\leq \la_k(\tilde \Om)$.}
\end{theorem}

For every $x_1\in \R$ and $r>0$, we define the \emph{strip} centered in $x_1$ of width $2r$ orthogonal to $\R e_1$ by
\[
S_r(x_1):=[-r+x_1,r+x_1]\times \R^{N-1}.
\] 
Its topological boundary is $ \partial S_r:=\{-r+x_1,r+x_1\}\times \R^{N-1}$. If $x_1=0$, we simply denote $S_r$ instead of $S_r(0)$.

The main idea of the following lemma is inspired from \cite{ac} and was also used  in~\cite{dpv14}, \cite{blak}, and  ~\cite{bbv} under different settings. We point out that here we do not use optimality, but only an inequality between two fixed domains. 
\begin{lemma}\label{caffarelli}
For all $c>0$, there exist $C_0,r_0 >0$, with $C_0r_0\leq \min{\left\{\frac c2, \frac{1}{2K}\right\}}$ such that if for some $r\leq r_0$ the function $w$ is not identically zero in  $S_r$ and 
\begin{equation}\label{ipcaf}
E(\tilde \Om)+c|\tilde \Om|\leq E(\tilde \Om\setminus S_{r})+c|\tilde \Om\setminus S_{r}|,
\end{equation}
then 
\begin{equation}\label{ipcaf2}
\max_{S_{2r}}{w_{\tilde \Om}}\ge C_0 r.
\end{equation}

\end{lemma}
\begin{proof}
Below, we denote  $w: =w_{\tilde \Om}$ and  $\eps:=\max_{S_{2r}}{w}$ and introduce the  function $\eta\colon \R^N\rightarrow \R^+$ :
\begin{equation}\label{eta}
\left\{
\begin{split}
\eta&=0 \quad \mbox{in }S_r, \qquad \eta=\eps\quad\mbox{in }\R^N\setminus S_{2r},\\
-\Delta \eta&=1\quad \mbox{in } S_{2r}\setminus S_r,\\
\eta&=0\quad \mbox{on }\partial S_r,\\
\eta&=\eps\quad \mbox{on }\partial S_{2r}.
\end{split}
\right.
\end{equation}
Since the function $\min{\{w,\eta\}}:=w\wedge \eta$ belongs to $ H^1_0(\tilde \Om\setminus S_r)$ we get \[
E(\tilde \Om\setminus S_r)\leq \frac12 \int{|D(w\wedge \eta)|^2}dx-\int{w\wedge \eta}dx.
\]
Hypothesis  \eqref{ipcaf} gives \[
\frac12\int{|Dw|^2}dx-\int{w}dx+c|S_r\cap \tilde \Om|\leq \frac12 \int{|D(w\wedge \eta)|^2}dx-\int{w\wedge \eta}dx.
\]
Since $\eps\leq C_0r_0\leq \frac{c}{2}$ we get $w\wedge \eta =w$ in $\tilde \Om\setminus S_{2r}$. Denoting the outer unit normal to a set by $\nu$, \[
\begin{split}
\frac12\int_{S_r}{|Dw|^2}dx+\frac{c}{2}|S_r\cap \tilde \Om|&\leq \frac12\int_{S_r}{|Dw|^2}dx-\int_{S_r}{w}dx+c|S_r\cap \tilde \Om|\\
&\leq\frac12\int_{S_{2r}\setminus S_r}{(|D(w\wedge \eta)|^2-|Dw|^2)}dx-\int_{S_{2r}\setminus S_r}{(w\wedge \eta-w)}dx\\
&=\frac12\int_{S_{2r}\setminus S_r\cap \{w>\eta\}}{(|D\eta|^2-|Dw|^2)}dx-\int_{S_{2r}\setminus S_r}{(w-\eta)^+}dx\\
&\leq \int_{S_{2r}\setminus S_r\cap \{w>\eta\}}{-D\eta\cdot D(w-\eta)}dx-\int_{S_{2r}\setminus S_r}{(w-\eta)^+}dx\\
&=-\int_{\partial S_r}{\frac{\partial \eta}{\partial \nu}(w-\eta)^+}dx=|\eta'(r)|\int_{\partial S_r}{w\,d\mathcal{H}^{N-1}}.
\end{split}
\]
The following trace inequality holds:\[
\int_{\partial S_r}{w\,d\mathcal{H}^{N-1}}\leq C(N)\left(\frac{1}{r}\int_{S_r}{w}dx+\int_{S_r}{|Dw|}dx\right).
\]
By using  hypothesis~\eqref{ipcaf2} and the Cauchy-Schwarz inequality on the gradient term, we get to\[
\int_{\partial S_r}{w\,d\mathcal{H}^{N-1}}\leq C(N)\left((C_0+\frac12)|S_r\cap \tilde \Om|+\frac12\int_{S_r\cap \tilde \Om}{|Dw|^2}dx\right).
\]
If $\int_{S_r}{|Dw|^2}dx+|S_r\cap \tilde \Om|=0$ then $w=0$ in the ``strip'' $S_r.$
Otherwise $\int_{S_r}{|Dw|^2}dx+|S_r\cap \tilde \Om|>0$ and from the previous inequality we get \[
\min{\left\{\frac12, \frac{c}{2}\right\}}\leq |\eta'(r)|C(N)(C_0+1).
\]
Since $|\eta'(r)|=|C_0-r/2|$,  choosing $C_0$ and $r_0$ small enough we get a contradiction. We notice that the choice of these constants depends \emph{only} on $N$ and on $c$.
\end{proof}

The following corollary  can be proved in the very same way as  Lemma~\ref{caffarelli}.
\begin{corollary}\label{cafcor}
For all $c>0$ there exist $C_0,r_0>0$ with $C_0r_0\leq \min{\left\{\frac c2, \frac{1}{2K}\right\}}$ such that if for some $r\leq r_0$ and $x_1,\dots,x_n\in\R$ such that $S_{2r}(x_i)\cap S_{2r}(x_j)=\emptyset$ for all $i\not=j,$ it holds \[
\max{\left\{w_{\tilde \Om}(x)\;:\;x\in \cup_i{S_{2r}(x_i)}\right\}}\leq C_0r_0,
\]
then we have that 
\begin{equation}\label{bmd07}
E(\tilde \Om\setminus \cup_i{S_{r}(x_i)})+c|\tilde \Om\setminus \cup_i{S_r(x_i)}|\leq E(\tilde \Om)+c|\tilde \Om|.
\end{equation}
\end{corollary}
Here we outline the main idea for proving Theorem~\ref{main}. 
Let $c$ be as in Lemma~\ref{choicec} and $r_0,C_0$ be the constants from Lemma~\ref{caffarelli}, for that particular choice of $c$. We shall remove a finite number of strips $S_{r}(x_i)$ from the region where $w_{\tilde \Om}(x)\leq C_0r_0$ thus, following inequality \eqref{bmd07} and Lemma~\ref{choicec}, we can control the eigenvalues after rescaling. The control of the perimeter, will be done by a suitable choice of the position of the strips. Contrary to the construction in \cite{mp}, the new perimeter introduced by sectioning with hyperplanes does not depend on the $\H^{N-2}$ measure of the boundary of the sections.

Let  $l_0>0$ and  $n \in \N$. The value of $l_0$ will be precised below, in Lemma \ref{perbound}. Assume $x^i\in \R^N$ and $L_i>2l_0$,  $i=1,\dots, n$ are such that  $S_{2L_i}(x_1^i)\cap S_{2L_j}(x_1^j)=\emptyset$ if $i\not=j$.
 For every $t\in [0,l_0]$ we  define:\[
S(t):=\cup_{i=1}^n{S_{L_i-t}(x_1^i)}.
\]
For an open set of unit measure $\tilde \Om$, we denote $m(t):=|S(t)\cap \tilde \Om|$ the mass of the union of strips in $\tilde \Om$ and \[ \sigma(t):=\sum_{i=1}^n{\mathcal{H}^{N-1}(\tilde \Om\cap \{L_i-t,L_i+t\}\times \R^{N-1})},
\] 
the new perimeter introduced by the sections with the hyperplanes and\[
p(t)=\sum_{i=1}^n{\p(\tilde \Om\cap(L_i-t,L_i+t)\times \R^{N-1})}-\sigma(t),\] 
the perimeter of  $\tilde \Om$ inside the strips.
We denote the rescaled set, \[
\Om(t):=(1-m(t))^{-1/N}(\tilde \Om\setminus S(t)).
\]
 
\begin{lemma}\label{perbound}
Given $P>0$ and an open set $\tilde \Om$ of unit measure, with $\p(\tilde \Om)\leq P$,
there exist two constants $l_0$ and $\widehat m$, depending \emph{only} on $P$ and the dimension $N$, such that if $m(l_0)\leq \widehat m$ then there exists $t\in[0,l_0]$ such that  $\p (\Om(t))\leq \p(\tilde \Om)$.
\end{lemma}
\begin{proof}
First of all, we notice that, by definition, $t\mapsto m(t)$ is a nonincreasing function and for a.e. $t\in (0,l_0)$, we have that $\sigma(t)=-m'(t)$.
If for every $t\in[0,l_0]$ we would have $\p(\Om(t))> \p(\tilde \Om)$, we get:\[
\p (\tilde \Om)-p(t)+\sigma(t)\geq \p (\tilde \Om)(1-m(t))^{\frac{N-1}{N}}.
\]
There exists a constant $\widehat m$ (depending only on $P,N$), such that if $m(t)\leq \widehat m$, then
 \[
(1-m(t))^{\frac{N-1}{N}}\geq 1-\frac{m(t)^{\frac{N-1}{N}}}{2P}\geq 1-\frac{m(t)^{\frac{N-1}{N}}}{2\p(\tilde \Om)}.
\]
Putting the above inequalities together and using  the isoperimetric inequality for the set $S(t)$,\[
\p(\tilde \Om)+2\sigma(t)\geq \p(\tilde \Om)-\frac{m(t)^{\frac{N-1}{N}}}{2}+p(t)+\sigma(t)\geq \p(\tilde \Om)-\frac{m(t)^{\frac{N-1}{N}}}{2}+N\omega_N^{1/N}m(t)^{\frac{N-1}{N}}.
\] 
Since  $2N\omega_N^{1/N}- 1>0$, we obtain:\[
-m'(t)\geq (2N\omega_N^{1/N}-1)\frac{m(t)^{\frac{N-1}{N}}}{4}.
\]
By integrating on $[0, l_0]$ we get
$$ m^{1/N}(0)- m^{1/N}(l_0)\ge (2\omega_N^{1/N}-1)\frac{l_0}{4N}.$$
Since $m(0)= \widehat m$ and $m(l_0)\ge 0$, choosing $l_0 > \frac{4N}{2\omega_N^{1/N}-1} \widehat m^{1/N}$ we get a contradiction.
\end{proof}

\begin{remark}\label{mtilde}
If we denote by $A$ a subset of $\tilde \Om$ with $\max_A{w_{\tilde \Om}}\leq C_0r_0$ then, having in mind~\eqref{vdb}, if 
$\widehat m$ is small enough (depending only on $C_0$ and $r_0$) we get
 \[
\la_1( A)(1-\widehat m)^{2/N}\geq \frac{1}{2C_0r_0}\geq K. 
\]
\end{remark}
\begin{remark}\label{positiveenergy}
Thanks to the choice of $C_0,r_0$ made in Lemma~\ref{caffarelli}, we deduce that if $A\subset \tilde \Om$ is such that  $\max_{A}{w_{\tilde \Om}}\leq C_0r_0$, then $E(A)+c|A|\geq 0$.
Indeed, using the monotonicity of the torsion function: \[
E(A)+c|A|=-\frac12\int{w_A}dx+c|A|\geq-\frac12\int_{A}{w_{\tilde \Om}}dx+c|A|\geq -\frac{C_0r_0|A|}{2}+c|A|\geq 0,
\]
since $C_0r_0\leq 2c$ from the hypotheses of Lemma~\ref{caffarelli}.
\end{remark}

We are now in position to prove the main result of this section. 
\begin{proof}[Proof of Theorem \ref{main}] We fix the constant $c$ such that Lemma~\ref{choicec} is satisfied, we get $C_0,r_0$ from Lemma~\ref{caffarelli} and we fix a constant $\widehat m$ that works both for Lemma~\ref{perbound} and for Remark~\ref{mtilde}. 
For simplicity we rename $w=w_{\tilde \Om}$. 
The region where  $w(x)\geq C_0r_0$    is contained in a finite union of  strips with width $4r_0$. Indeed, we define \[
X_0:=\left\{x_1\in \R\;:\;\max_{S_{2r_0}(x_1)}{w}\geq C_0r_0\right\}, \qquad \widetilde X:=\bigcup\Big\{S_{2r_0}(t) : t\in X_0\Big\}.
\]
From  Lemma~\ref{density} and the Saint Venant inequality~\eqref{saintvenant} the set $\widetilde X$ is contained  in the union of at most $n=n(r_0,N)$ of disjoint strips (each of width at least $4r_0$). Let us call $X$ the projection of $\widetilde X$ on $\R e_1$.

The set $\R\setminus  X$ is a  finite  union of disjoint segments and of the infinite intervals at $\pm\infty$, say 
\[
\R\setminus X=(-\infty,b_0)\cup\left[ \bigcup_{i=1}^{n}{(a_i,b_i)}\right]\cup (a_{n+1},\infty).
\]
If a segment $(a_i,b_i)$ has a length less than  or equal to $8r_0+2l_0$, we shall ignore it in our further construction and just add the corresponding strip to the set $\tilde X$ and renumber the index $i$ if necessary. The total length of  those such segments is at most  $ n(8r_0+2l_0)$. 

Therefore, we shall assume in the sequel that all segments $(a_i,b_i)$ have a length 
 greater than  $8r_0+2l_0$.
We denote $\overline a_i= a_i + (4r_0+l_0)$, $\overline b_i= b_i -( 4r_0+l_0)$ and
 $$Y=  \left[\bigcup_{i=1}^{n+1}(a_i, \overline a_i)\right]\cup\left[ \bigcup_{i=0}^{n}(\overline b_i,  b_i)\right].$$

In order to highlight the main idea, let us assume in a first instance that
\begin{equation}\label{masssmall}
| \left ( Y \times \R^{N-1} \right ) \cap \tilde \Om | \leq \widehat m.
\end{equation}
If we are in this situation, we perform a simultaneous ``cut'' as in Lemma~\ref{caffarelli}, removing the following union of strips:\[
S_t:= S_{r_0}(b_0-2r_0-t)\bigcup   S_{r_0}(a_i+2r_0+t)\bigcup S_{r_0}(b_i-2r_0-t)\bigcup S_{r_0}(a_{n+1}+2r_0+t),
\]
for every $t \in [0, l_0]$. 

Following the assumption \eqref{masssmall} and Lemma~\ref{perbound}, there exists a value $t$ such that the perimeter of the rescaled set $|\tilde \Om \setminus \overline S_t|^{- \frac 1N} (\tilde \Om \setminus \overline S_t)$ is at most $\p (\tilde \Om)$. Moreover, from the choice of $c$ and  Lemma~\ref{caffarelli}, all the eigenvalues less than $K$ of the rescaled set are not greater than the ones on $\tilde \Om$. 

In order to handle the diameter of the rescaled set, we replace all the connected components having a projection on $\R e_1$ disjoint from $X$ by one ball, such that the  volume remains unchanged. In this way, the perimeter does not increase, while the low part of the spectrum (below $K$) can only decrease, since the first eigenvalue of every such a connected component is not smaller than $1/(C_0r_0)\geq 2K$ (see Remark \ref{mtilde}). 

It is clear that the new set  satisfies the diameter bound:
\[
\diam_{e_1}(\Om)\leq \diam_{e_1}(\widehat \Om)(1-\widehat m)^{-1/N}\leq 2\Big( \H^1 (X)+n(8r_0+2l_0)+2r_0(n+2)\Big) + 2 \omega_N^{- \frac {1}{N}}.
\]

If assumption \eqref{masssmall} does not hold, we can not apply directly Lemma~\ref{perbound}.
Let $p\in \N$ depending only on $P$ and the dimension, be such that $\frac 1p \le \widehat m < \frac{1}{p-1}$. If
$$a_i +p (4r_0+l_0)> b_i-p (4r_0+l_0),$$
we ignore this strip and add it to $X$, renumbering the index $i$ if necessary. There exists $s \in [0,p-1]$ such that replacing simultaneously all $a_i$ with $a_i + s (4r_0+l_0)$ and $b_i$ with $b_i - s (4r_0+l_0)$ the assumption  \eqref{masssmall} is satisfied
and so we finish the proof, adding at worst $4np (4r_0+l_0)$ to the diameter. 
\end{proof}

\begin{remark}
Since the choice of the direction $e_1$ was arbitrary, we can repeat all the process of the proof of Theorem~\ref{main} for all the coordinate direction, finding a set which has diameter bounded in all directions, unit measure, better eigenvalues than $\tilde \Om$ up to level $k$ and perimeter lower than $\tilde \Om$.
\end{remark}

\noindent{\bf Acknowledgments.}  This work was initiated during the first author stay at the Isaac Newton Institute for Mathematical Sciences Cambridge in the programme ``Free boundary problems and related topics", as an answer to a question raised by Michiel Van den Berg. The first author also acknowledges the support of  ANR--12--BS01--0007 Optiform. The work of the second author has been supported by the ERC Starting Grant n.\ 258685 ``AnOptSetCon''.


\begin{thebibliography}{9}

\bibitem{ac} H.~W.~Alt, L.~A.~Caffarelli, Existence and regularity for a minimum problem with
free boundary, J. Reine Angew. Math., {\bf 325} (1981),  105--144.

\bibitem{A} M.S. Ashbaugh, The universal eigenvalue bounds of Payne-P\"olya-Weinberger, Hile-Protter, and H C Yang, Proc. Indian Acad. Sci. (Math. Sci.) {\bf 112} (1) (2002), 3--30.

\bibitem{vb1} M.~van den Berg, Estimates for the torsion function and Sobolev constants, Potential Anal., {\bf 36} (4) (2011) 607--616.  

\bibitem{vb} M.~van den Berg, On the minimization of Dirichlet eigenvalues, preprint available at {\tt  arXiv:1405.0127}.

\bibitem{bvb} M.~van den Berg, D.~Bucur, On the torsion function with Robin or Dirichlet boundary conditions, Journal of Functional Analysis (2013), http://dx.doi.org/10.1016/j.jfa.2013.07.007.

\bibitem{bhp05} T.~Brian\c{c}on, M.~Hayouni, M.~Pierre, Lipschitz continuity of state functions in some optimal shaping, \emph{ Calc. Var. PDE}, {\bf 23} (1) (2005), 13--32.

\bibitem{blak} D. Bucur, Minimization of the $k$-th eigenvalue of the Dirichlet Laplacian, Arch. Ration. Mech. Anal. {\bf 206} (2012) 1073--1083.

\bibitem{bbv} D.~Bucur, G.~Buttazzo, B.~Velichkov, Spectral optimization for potentials and measures, SIAM Journal of Mathematical Analysis, to appear (2014).

\bibitem{bb} D.~Bucur, G.~Buttazzo, Variational methods in shape optimization problems, Progress in nonlinear differential equations and their applications, Birkh\"auser Verlag, Boston (2005).

\bibitem{bmpv14} D.~Bucur, D.~Mazzoleni, A.~Pratelli, B.~Velichkov, Lipschitz regularity of the eigenfunctions on optimal domains, to appear on Arch. Ration. Mech. Anal. (2014).


\bibitem{BM} G. Buttazzo, G. Dal Maso, An existence result for a class of shape optimization problems, Arch. Ration. Mech. Anal. {\bf 122} (1993), 183--195.

\bibitem{dpv14} G. De Philippis, B.  Velichkov, Existence and regularity of minimizers for some spectral functionals with perimeter constraint. Appl. Math. Optim. 69 (2014), no. 2, 199--231.
 
\bibitem{H} A. Henrot, Extremum Problems for Eigenvalues of Elliptic Operators, Frontiers in Mathematics, Birkh\"auser Verlag, Basel (2006).

\bibitem{hayouni} M.~Hayouni, Lipschitz continuity of a state function in a shape optimization problem, J. Convex Anal., {\bf 6} (1999) (1), 71--90.

\bibitem{hp} A.~Henrot, M.~Pierre, Variation et optimisation de formes, Math\'ematiques et Applications {\bf 48}, Springer (2005).

\bibitem{ly83} P.~Li, S.-T.~Yau, On the Schr\"odinger equation and the eigenvalue problem,
Comm. Math. Phys. {\bf 88} (1983), no. 3, 309--318. 

\bibitem{mp} D.~Mazzoleni, A.~Pratelli, Existence of minimizers for spectral problems,  J. Math. Pures Appl., {\bf 100} (2013), 433--453.

\bibitem{talenti} G.~Talenti, Elliptic equations and rearrangements, Ann. Scuola Norm. Sup. Pisa
Cl. Sci., {\bf 3} (1976) (4), 697--718.


\end{thebibliography}
\end{document}